\documentclass[reqno,12pt]{amsart}

\usepackage{amsmath, amsthm, amsfonts, amssymb}
\input{mathrsfs.sty}

\usepackage{amsmath}
\usepackage{amsfonts}
\usepackage{amssymb}
\usepackage{eufrak}
\usepackage{amscd}
\usepackage{amsthm}
\usepackage{amstext}
\usepackage[all]{xy}

   \newcommand{\Hom}{\operatorname{Hom}}

   \theoremstyle{plain}
   \newtheorem{thm}{Theorem}[section]
   \newtheorem{prop}[thm]{Proposition}
   \newtheorem{lem}[thm]{Lemma}
   \newtheorem{cor}[thm]{Corollary}
   \theoremstyle{definition}
   
   \newtheorem{defn}[thm]{Definition}
   \newtheorem{example}[thm]{Example}
   \theoremstyle{remark}
   
   \newtheorem{remark}[thm]{Remark}

 \numberwithin{equation}{section}

\author{V. Manuilov}

\date{}

\address{Moscow Center for Fundamental and Applied Mathematics {\rm and} Moscow State University,
Leninskie Gory 1, Moscow, 
119991, Russia}

\email{manuilov@mech.math.msu.su}

\thanks{The author acknowledges partial support by the RFBR grant.}

\title{Metrics on doubles as an inverse semigroup II}

\begin{document}

\begin{abstract}
We have shown recently that, given a metric space $X$, the coarse equivalence classes of metrics on the two copies of $X$ form an inverse semigroup $M(X)$. Here we give a description of the set $E(M(X))$ of idempotents of this inverse semigroup and of its Stone dual space $\widehat X$. We also construct $\sigma$-additive measures on $\widehat X$ from finitely additive probability measures on $X$ that vanish on bounded subsets. 

\end{abstract}

\maketitle

\section*{Introduction}

Given metric spaces $X$ and $Y$, a metric $d$ on $X\sqcup Y$ that extends the metrics on $X$ and $Y$, depends only on the values $d(x,y)$, $x\in X$, $y\in Y$, but it may be not easy to check which functions $d:X\times Y\to (0,\infty)$ determine a metric on $X\sqcup Y$: one has to check infinitely many triangle inequalities. The problem of description of all such extended metrics is difficult due to the lack of a nice algebraic structure on the set of metrics. It turns out that, passing to quasi-equivalence (or coarse equivalence) of metrics, we can find some algebraic structure. First, we can define a composition: if $d$ is a metric on $X\sqcup Y$, and if $\rho$ is a metric on $Y\sqcup Z$ then the formula $(\rho\circ d)(x,z)=\inf_{y\in Y}[d(x,y)+\rho(y,z)]$ defines a metric $\rho d$ on $X\sqcup Z$. The idea to consider equivalence classes of metrics on the disjoint union of two spaces as morphisms from one space to another was suggested in \cite{Manuilov-Morphisms}.
 
It was a surprise for us to discover that in the case $Y=X$, there is a nice algebraic structure on the set $M^q(X)$ (resp., $M^c(X)$) of quasi-equivalence (resp., of coarse equivalence) classes of metrics on the double $X\sqcup X$: they form an {\it inverse} semigroup with respect to this composition \cite{M}. 

Recall that a semigroup $S$ is an inverse semigroup if for any $u\in S$ there exists a unique $v\in S$ such that $u=uvu$ and $v=vuv$ \cite{Lawson}. Philosophically, inverse semigroups describe local symmetries in a similar way as groups describe global symmetries, and technically, the construction of the (reduced) group $C^*$-algebra of a group generalizes to that of the (reduced) inverse semigroup $C^*$-algebra \cite{Paterson}. 

Any two idempotents of an inverse semigroup $S$ commute, and the semilattice $E(S)$ of all idempotents of $S$ generates a commutative $C^*$-algebra. Our aim is to get a better understanding of the Stone dual 
spaces $\widehat X^q$ of $E(M^q(X))$, and $\widehat X^c$ of $E(M^c(X))$.

It turns out that the coarse equivalence is better suited for study of the inverse semigroup from metrics on doubles, so we can give more detailed results for the semigroup $M^c(X)$.

We give several descriptions of $E(M^q(X))$ and $E(M^c(X))$, provide a description of the spaces $\widehat X^q$ and $\widehat X^c$ dual to $E(M^q(X))$ and to $E(M^c(X))$, respectively, and, under certain restrictions, we describe a dense set of $\widehat X^c$ in terms of free ultrafilters on $X$. We also construct $\sigma$-additive measures on $\widehat X^c$ from finitely additive probability measures on $X$ that vanish on bounded subsets. 
\bigskip

\section{The inverse semigroup from a metric space}

We begin with some basic facts from \cite{M}. 
Let $X$ be a metric space with a fixed metric $d_X$. 

\begin{defn}
A {\it double} of $X$ is a metric space $X\times\{0,1\}$ with a metric $d$ such that 
\begin{itemize}
\item[(d1)]
the restriction of $d$ on each copy of $X$ in
$X\times\{0,1\}$ equals $d_X$; 
\item[(d2)]
the distance between the two copies of $X$ is non-zero.
\end{itemize}

Let $\mathcal M(X)$ denote the set of all such metrics.

\end{defn}

We identify $X$ with $X\times\{0\}$, and write $X'$ for $X\times\{1\}$. Similarly, we write 
$x$ for $(x,0)$ and $x'$ for $(x,1)$, $x\in X$. 
Note that metrics on a double of $X$ may differ only when two points lie in different copies of $X$.

Recall that two metrics, $d_1$, $d_2$, on the double of $X$ are {\it quasi-equivalent} if there exist $\alpha>0,\beta\geq 1$ such that 
$$
-\alpha+\frac{1}{\beta} d_1(x,y')\leq d_2(x,y')\leq \alpha+\beta d_1(x,y')
$$
for any $x,y\in X$. They are {\it coarse equivalent} if there exist monotone functions $\varphi,\psi:[0,\infty)\to[0,\infty)$ with $\lim_{t\to\infty}\varphi(t)=\psi(t)=\infty$ such that 
$$
\varphi(d_1(x,y')\leq d_2(x,y')\leq\psi(d_1(x,y'))
$$
for any $x,y\in X$. 
If $d_1$ and $d_2$ are quasi-equivalent (resp., coarse equivalent) then we write
$d_1\sim_q d_2$ (resp., $d_1\sim_c d_2$), and the quasi-equivalence (resp., coarse equivalence) class of $d$ is denoted by $[d]_q$ (resp., by $[d]_c$). If the kind of equivalence doesn't matter then we write $d_1\sim d_2$ and $[d]$. The set of quasi-equivalence (resp., of coarse equivalence) of metrics in $\mathcal M(X)$ we denote by $M^q(X)$ (resp., by $M^c(X)$). When the kind of equivalence doesn't matter, we write $M(X)$ for either of them. 

Similarly, we consider two kinds of equivalence on the set of positive-valued functions on $X$:
Two functions, $f,g:X\to(0,\infty)$ are quasi-equivalent, $f\sim_q g$, if there exist $\alpha>0,\beta\geq 1$ such that 
$$
-\alpha+\frac{1}{\beta} f(x)\leq g(x)\leq \alpha+\beta f(x)
$$
for any $x\in X$. They are coarse equivalent, $f\sim_c g$, if there exist monotone functions $\varphi,\psi:[0,\infty)\to[0,\infty)$ with $\lim_{t\to\infty}\varphi(t)=\psi(t)=\infty$ such that 
$$
\varphi(f(x)\leq g(x)\leq\psi(f(x))
$$
for any $x\in X$. 

For a metric $d\in\mathcal M(X)$ define the adjoint metric $d^*$ by setting $d^*(x,y')=d(y,x')$, $x,y\in X$.
A class $s\in M(X)$ is selfadjoint if there exists $d\in s$ such that $d^*=d$. 

The following results were proved in \cite{M}:

\begin{thm}\label{projections}
Let $d^*=d$ be a metric on the double of $X$. Then $[d^2]_q=[d]_q$ if and only if there exist $\alpha\geq 0$, $\beta\geq 1$ such that 
$-\alpha+\frac{1}{\beta} d(x,x')\leq d(x,X')$ for any $x\in X$.

\end{thm}

\begin{prop}\label{xx'}
Let $d$, $\rho$ be two idempotent metrics on the double of $X$, $\rho^*=\rho$, $d^*=d$. Then $\rho\sim d$ if and only if
the functions $x\mapsto \rho(x,x')$ and $x\mapsto d(x,x')$ are equivalent.

\end{prop}

\section{Description of projections in $M(X)$ in terms of expanding sequences}

For  a subset $A\subset X$ we denote by $N_r(A)$ the $r$-neighborhood of $A$, i.e. 
$$
N_r(A)=\{x\in X:d_X(x,A)\leq r\}.
$$
The sequence $\mathcal A=\{A_n\}_{n\in\mathbb N}$, where $A_n\subset X$, is an {\it expanding sequence} if it satisfies 
\begin{itemize}
\item[(e1)]
$A_n$ is not empty for some $n\in\mathbb N$;
\item[(e2)]
$N_{1/2}(A_n)\subset A_{n+1}$ for any $n\in\mathbb N$, for which $A_n$ is non-empty.
\end{itemize}
Note that $\cup_{n\in\mathbb N}A_n=X$.

A special class of expanding sequences is given by subsets of $X$. For $A\subset X$, set $\mathcal E_A=\{A_n\}_{n\in\mathbb N}$, where $A_n=N_{n/2}(A)$.

Let $\delta:X\to[1,\infty)$ be any function (not necessarily continuous) on $X$. We shall write $\delta(u,u')$ instead of $\delta(u)$ for $u\in X$ to show that this function measures distance in some sense.

Define a metric on the double of $X$ by 
$$
d(x,y')=\inf_{u\in X}[d_X(x,u)+\delta(u,u')+d_X(u,y)].
$$

\begin{lem}\label{triangle1}
$d$ is a metric for any function $\delta$.

\end{lem}
\begin{proof}
It suffices to check the two triangle inequalities.

1. Let $x_1,x_2,y\in X$.
For any $u_1,u_2\in X$ we have
\begin{eqnarray*}
d_X(x_1,x_2)&\leq&d_X(x_1,u_1)+d_X(u_1,y)+d_X(y,u_2)+d_X(u_2,x_2)\\
&\leq&[d_X(x_1,u_1)+\delta(u_1,u'_1)+d_X(u_1,y)]+[d_X(y,u_2)+\delta(u_2,u'_2)+d_X(u_2,x_2)],
\end{eqnarray*}
hence, passing to the infimum over $u_1$ and $u_2$, we obtain
$$
d_X(x_1,x_2)\leq d(x_1,y')+d(x_2,y').
$$

2. Take $\varepsilon>0$, and let $\bar u_2$ satisfy 
$$
d(x_2,y')\geq [d_X(x_2,\bar u_2)+\delta(\bar u_2,\bar u'_2)+d_X(\bar u_2,y)]-\varepsilon.
$$
Then 
\begin{eqnarray*}
d(x_1,y')&\leq&[d_X(x_1,\bar u_2)+\delta(\bar u_2,\bar u'_2)+d_X(\bar u_2,y)]\\
&\leq& d_X(x_1,x_2)+d_X(x_2,\bar u_2)+\delta(\bar u_2,\bar u'_2)+d_X(\bar u_2,y)\\
&\leq& d_X(x_1,x_2)+d(x_2,y')+\varepsilon. 
\end{eqnarray*}

As $\varepsilon$ is arbitrary, we conclude that
$$
d(x_1,y')\leq d_X(x_1,x_2)+d(x_2,y').
$$

\end{proof}

Obviously, $d^*=d$. We shall call selfadjoint idempotents {\it projections}.

\begin{lem}\label{proj1}
$[d]$ is a projection in $M(X)$ for any function $\delta$.

\end{lem}
\begin{proof}
It follows from Theorem \ref{projections} and from the estimate
\begin{eqnarray*}
2d(x,X')&=&2\inf_{y,u\in X}[d_X(x,u)+\delta(u,u')+d_X(u,y)]\\
&=&2\inf_{u\in X}[d_X(x,u)+\delta(u,u')]\\
&\geq& \inf_{u\in X}[2d_X(x,u)+\delta(u,u')]=d(x,x')
\end{eqnarray*}
that $[d]_q$ is a projection in $M^q(X)$. As the canonical map $M^q(X)\to M^c(X)$ is a homomorphism, $[d]_c$ is a projection in $M^c(X)$ as well. 

\end{proof}

Given an expanding sequence $\mathcal A=\{A_n\}_{n\in\mathbb N}$, define the function
$\delta=\delta_{\mathcal A}$ on $X$ by $\delta_{\mathcal A}(u,u')=1$ if $u\in A_1$, and $\delta_{\mathcal A}(u,u')=n+1$ if $u\in A_{n+1}\setminus A_n$, $n\in\mathbb N$. 
For the expanding sequence $\mathcal A$, set
\begin{equation}\label{inf1}
d_\mathcal A(x,y')=\inf_{u\in X}[d_X(x,u)+\delta_\mathcal A(u,u')+d_X(u,y)].
\end{equation}

The following statement gives the description of projections in $M^q(X)$.

\begin{thm}\label{equiv}
Let $d\in\mathcal M(X)$ be a selfadjoint metric such that $[d]_q$ is a projection in $M^q(X)$. Set $A_n=\{x\in X:d(x,x')\leq n\}$. Then $\mathcal A=\{A_n\}_{n\in\mathbb N}$ is an expanding sequence, and $d\sim_q d_\mathcal A$. 

\end{thm}
\begin{proof}
If $x\in N_{1/2}(A_n)$ then $d_X(x,A_n)\leq 1/2$, hence 
$$
d(x,x')\leq \inf_{u\in A_n}[d_X(x,u)+d(u,u')+d_X(u,x)]\leq \inf_{u\in A_n}[2d_X(x,u)+n]\leq n+1, 
$$
hence the sequence $\mathcal A=\{A_n\}_{n\in\mathbb N}$ is an expanding sequence. 
By Proposition \ref{xx'}, it suffices to show that the functions $d(x,x')$ and $d_\mathcal A(x,x')$ are quasi-equivalent.

Let $x\in A_n\setminus A_{n-1}$ for some $n\in\mathbb N$. Then $n-1 < d(x,x')\leq n$. For $u\in X$ set $d_X(x,u)=r$.
Then $u\in A_{n+2r}\setminus A_{n-1-2r}$, hence $n-1-2r\leq \delta_\mathcal A(u,u')\leq n+2r$. Thus,
$$
d_\mathcal A(x,x')=\inf_{u\in X}[2d_X(x,u)+\delta_\mathcal A(u,u')]\geq 2r+n-1-2r=n-1.
$$
On the other hand, taking $u=x$ in the right hand side of (\ref{inf1}) instead of the infimum, we see that
$$
d_\mathcal A(x,x')\leq \delta_\mathcal A(x,x')\leq n.
$$ 
Thus, for $x\in X$ satisfying $n-1<d(x,x')\leq n$ we have
$n-1\leq d_\mathcal A(x,x')\leq n$.

\end{proof}

\begin{lem}\label{le2}
Let $d_1,d_2\in\mathcal M(X)$ be selfadjoint metrics, $A^{(i)}_n=\{x\in X:d_i(x,x')\leq n\}$, $i=1,2$. The following conditions are equivalent:
\begin{itemize}
\item[(a1)]
$[d_1]_q=[d_2]_q$;
\item[(a2)]
there exist $\alpha,\beta\in\mathbb N$ such that $A^{(1)}_n\subset A^{(2)}_{\beta n+\alpha}$ and $A^{(2)}_n\subset A^{(1)}_{\beta n+\alpha}$. 
\end{itemize}
Similarly, the following conditions are equivalent:
\begin{itemize}
\item[(b1)]
$[d_1]_c=[d_2]_c$;
\item[(b2)]
there exists a monotone function $\varphi$ on $[0,\infty)$ with $\lim_{t\to\infty}\varphi(t)=\infty$ such that $A^{(1)}_n\subset A^{(2)}_{\varphi(n)}$ and $A^{(2)}_n\subset A^{(1)}_{\varphi(n)}$.
\end{itemize}

\end{lem}
\begin{proof}
Obvious.

\end{proof}

Recall that the zero element $\mathbf 0\in M(X)$ is represented by the metric $d_{x_0}$, $d_{x_0}(x,y')=[d_X(x,x_0)+1+d_X(x_0,y)]$, $x,y\in X$, with any fixed $x_0\in X$ \cite{M}. The expanding sequence for $d_{x_0}$ is given by $\mathcal E_{\{x_0\}}$. We denote by $B_r(x_0)$ the ball of radius $r$ centered at $x_0$.  

\begin{lem}\label{0} 
Let $\mathcal A=\{A_n\}$, $\mathcal D=\{D_n\}$, $n\in\mathbb N$.
The following are equivalent:
\begin{itemize}
\item[(c1)]
$[d_{\mathcal A}d_{\mathcal D}]_q=\mathbf 0$; 
\item[(c2)]
there exists $\beta\geq 1$, $\alpha\geq 0$ such that for any $n\in\mathbb N$, $A_n\cap D_n\subset B_{\beta n+\alpha}(x_0)$.
\end{itemize}
Similarly, the following are equivalent:
\begin{itemize}
\item[(d1)]
$[d_{\mathcal A}d_{\mathcal D}]_c=\mathbf 0$;
\item[(d2)]
for any $n\in\mathbb N$, the set $A_n\cap D_n$ is bounded.
\end{itemize}

\end{lem}
\begin{proof}
If $[d_{\mathcal A}d_{\mathcal D}]_q=\mathbf 0$ then there exist $\beta\geq 1$, $\alpha\geq 0$ such that
\begin{eqnarray*}
d_{\mathcal A}d_{\mathcal D}(x,x')&=&\inf_{u,y,v\in X}[d_X(x,u)+\delta_{\mathcal A}(u,u')+d_X(u,y)+d_X(y,v)+\delta_{\mathcal D}(v,v')+d_X(v,x)]\\
&\geq& \frac{1}{\beta}[2d_X(x,x_0)+1]-\alpha=\frac{1}{\beta}d_{x_0}(x,x')-\alpha.
\end{eqnarray*}

Taking $u=y=v=x$, we get 
$$
\delta_{\mathcal A}(x,x')+\delta_{\mathcal D}(x,x')\geq \frac{2}{\beta}d_X(x,x_0)-\alpha
$$
for any $x\in X$. 

Suppose that $x\in A_n\cap D_n$. Then $\delta_{\mathcal A}(x,x')\leq n$ and $\delta_{\mathcal D}(x,x')\leq n$, hence
$$
\frac{2}{\beta}d_X(x,x_0)-\alpha\leq 2n,
$$
$$
d_X(x,x_0)\leq \beta n+\alpha',
$$
where $\alpha'=\alpha\beta/2$,
thus $x\in B_{\beta n+\alpha'}(x_0)$.

Similarly, if $[d_{\mathcal A}d_{\mathcal D}]_c=\mathbf 0$ then there exists a monotone function $\varphi$ on $[0,\infty)$ with $\lim_{t\to\infty}\varphi(t)=\infty$ such that
$$
\delta_{\mathcal A}(x,x')+\delta_{\mathcal D}(x,x')\geq \varphi(d_X(x,x_0))
$$
for any $x\in X$. If $x\in A_n\cap D_n$ then $2n\geq \varphi(d_X(x,x_0))$, hence $d_X(x,x_0)\leq\varphi^{-1}(2n)$, and the set $A_n\cap D_n$ is bounded.

In the opposite direction, assume that there exist $\beta\geq 1$, $\alpha\geq 0$ such that $A_n\cap D_n\subset B_{\beta n+\alpha}(x_0)$ for any $n\in\mathbb N$. If 
$$
\beta n+\alpha\leq d_X(x,x_0)\leq \beta (n+1)+\alpha
$$ 
then $x$ cannot lie both in $A_n$ and $D_n$. Suppose $x\notin A_n$ holds.
Then, as in the proof of Theorem \ref{equiv}, $d_{\mathcal A}(x,x')\geq n-1$. The proof of Lemma \ref{proj1} shows that then $d_{\mathcal A}(x,X')\geq\frac{n-1}{2}$. Then $d_{\mathcal A}d_{\mathcal D}(x,x')\geq \frac{n-1}{2}$. The case when $x\notin D_n$ implies the same conclusion. Thus, for any $x\in X$ there is $n\in\mathbb N$ such that
$$
d_X(x,x_0)\leq\beta (n+1)+\alpha; \quad d_{\mathcal A}d_{\mathcal D}(x,x')\geq \frac{n-1}{2},
$$ 
which implies that for certain $\beta'\geq 1$, $\alpha'\geq 0$,
$$
d_{x_0}(x,x')\leq \frac{1}{\beta'}d_{\mathcal A}d_{\mathcal D}(x,x')-\alpha'
$$
for any $x\in X$.

The triangle inequality implies that $2d_X(x,x_0)+d(x_0,x'_0)\geq d(x,x')$ for any metric $d\in\mathcal M(X)$, hence $d_{x_0}(x,x')\geq \frac{1}{2}d(x,x')-\alpha''$, where $\alpha''=d(x_0,x'_0)$, thus $d_{x_0}\sim_q d_{\mathcal A}d_{\mathcal D}$.

Similarly, assume that the set $A_n\cap D_n$ is bounded for any $n\in\mathbb N$, i.e. there exists $x_0\in X$ and for any $n$ there exists $m(n)$ such that $A_n\cap D_n\subset B_{m(\frac{n-2}{2})}(x_0)$. Let $x\in X$ satisfy 
$$
\textstyle m(\frac{n-2}{2}) < d(x,x_0)\leq m(\frac{n-1}{2}).
$$
Then $x$ cannot lie both in $A_n$ and $D_n$, hence $d_{\mathcal A}d_{\mathcal D}(x,x')\geq\frac{n-1}{2}$. 
We may think of $m$ as an increasing function with $\lim_{t\to\infty}m(t)=\infty$, then
$$
\textstyle m^{-1}(d(x,x_0))\leq \frac{n-1}{2}\leq d_{\mathcal A}d_{\mathcal D}(x,x'),
$$
for any $x\in X$, hence $d_{\mathcal A}d_{\mathcal D}\sim_c d_{x_0}$.

\end{proof}

\section{Two types of projections}

By Theorem \ref{equiv}, any projection $e\in E(M(X))$ is equivalent to a projection of the form $d_{\mathcal A}$ with an expanding sequence $\mathcal A=\{A_n\}_{n\in\mathbb N}$. There are two possibilities for the expanding sequence $\mathcal A$: if for any $m\in\mathbb N$ there exist $n,k\in\mathbb N$ such that $A_m\subset N_k(A_n)$ then we say that $d_{\mathcal A}$ is of type I, otherwise it is of type II. By Lemma \ref{le2}, the type is well-defined, i.e. does not depend on the representative of the equivalence class $e$. If $[d]_c$ is of type I then there exists $A\subset X$ such that $d$ is coarsely equivalent to the metric $b_A\in\mathcal M(X)$ defined by $b_A(x,y')=d_X(x,A)+1+d_X(y,A)$, $x,y\in X$.

Note that the product of projections of the same type need not be a projection of the same type. For type II this is easy: one can take any two projections of type II with product $\mathbf 0$ (cf. Lemma \ref{0}). For the product of type I projections see Example \ref{typeI}. Nevertheless, we can show that at least some algebraic property holds for type I projections.

\begin{lem}\label{A-B}
Let $d\in\mathcal M(X)$, $A\subset X$, $d^*d\sim_c d_{\mathcal E_A}$. Then there exists $B\subset X$ such that $dd^*\sim_c d_{\mathcal E_B}$.

\end{lem}
\begin{proof}
By definition, $d^*d(x,x')=2d(x,X')$ and $dd^*(x,x')=2d(x',X)$ for any $x\in X$. Let 
$$
A_n=\{x\in X:d(x,X')\leq n\},\quad B_n=\{x\in X:d(x',X)\leq n\}. 
$$
Then $\{A_n\}$ and $\{B_n\}$ are expanding sequences for $\frac{1}{2}d^*d$ and $\frac{1}{2}dd^*$ respectively. 
As $d^*d\sim_c d_{\mathcal E_A}$, there exists $n_0\in\mathbb N$ and a monotone function $\varphi$ with $\lim_{n\to\infty}\varphi(n)=\infty$ such that $A\subset A_n\subset N_{\varphi(n)}(A)$ for any $n\geq n_0$.

We shall prove by induction that for each $n\in\mathbb N$ there exists $k\in\mathbb N$ such that $B_n\subset N_k(B_{n_0})$, $n\geq n_0$.
This is obvious for $n=n_0$, so assume that this holds for $n$ and let us show that this holds for $n+1$. Let $x\in B_{n+1}$.
Then there exists $y\in X$ such that $d(x',y)\leq n+2$. By definition, $y\in A_{n+2}$, and there exists $m\in\mathbb N$ such that $A_{n+2}\subset N_m(A)$. Then there exists $z\in A\subset A_{n_0}$ such that $d_X(y,z)\leq m+1$. Then there exists $t\in X$ such that $d(z',t)\leq n_0$, which means that $t\in B_{n_0}$. Then 
$$
d_X(x,t)\leq d(x,y')+d_X(y',z')+d(z',t)\leq n+2+m+1+n_0=k, 
$$
i.e. $x\in N_k(B_{n_0})$.   

\end{proof}

\begin{remark}
It would be natural to ask if a stronger statement holds: if $e\in E(M(X))$ is a projection of type I then $ses^*$ is of type I for any $s\in M(X)$. The Example \ref{typeI} shows that this is not true: indeed, take $A,B\subset X$ such that $[d_{\mathcal E_A}][d_{\mathcal E_B}]$ is not of type I, and take $e=[d_{\mathcal E_A}]$, $s=s^*=[d_{\mathcal E_B}]$. As projections commute, we have $ses^*=se$.  

\end{remark}
 
Recall that there is a partial order on $\mathcal M(X)$: $d_1\preceq d_2$ if $d_1(x,y')\geq d_2(x,y')$ for any $x,y\in X$, and that this partial order passes to the quotients $M^{(\cdot)}(X)$. If $[d_1]$, $[d_2]$ are projections then $[d_1]\preceq[d_2]$ is equivalent to $[d_1][d_2]=[d_1]$.

It is clear that if $\mathcal A=\{A_n\}_{n\in\mathbb N}$ is an expanding sequence then $d_{\mathcal E_{A_n}}\preceq d_{\mathcal A}$ for any $n\in\mathbb N$.

\begin{lem}\label{compar}
Let $d_{\mathcal E_B}\preceq d_{\mathcal A}$. Then there exists $n\in\mathbb N$ such that $B\subset A_n$.

\end{lem}
\begin{proof}
If $x\in B$ then $d_{\mathcal E_B}(x,x')=d_X(x,B)+1=1$, so, if $d_{\mathcal E_B}\preceq d_{\mathcal A}$ then the map $x\mapsto d_{\mathcal A}(x,x')$ is bounded on $B$ from above by some $n\in\mathbb N$, i.e. $d_{\mathcal A}(x,x')\leq n$. But this shows that $x\in A_n$, hence $B\subset A_n$. 

\end{proof}

\section{Description of projections in terms of functions on $X$}

Let $\mathcal C_m(X)$ be the set of all continuous functions $f$ on $X$ taking values in $[0,\infty)$ such that
\begin{itemize}
\item[(f1)]
there exists $\varepsilon>0$ such that $f(x)\geq\varepsilon$ for any $x\in X$;
\item[(f2)]
$|f(x)-f(y)|\leq 2d_X(x,y)$ for any $x,y\in X$.

\end{itemize}

We have two kinds of equivalence on $\mathcal C_m(X)$, quasi-equivalence and coarse equivalence. Let $C_m^q(X)$ (resp., $C_m^c(X)$) be the set of quasi-equivalence (resp., of coarse equivalence) classes of $\mathcal C_m(X)$. If the kind of equivalence doesn't matter then we write just $C_m(X)$ for any of them.

We say that $[f]_q\preceq[g]_q$ if there exist $\alpha\geq 0$, $\beta\geq 1$ such that $g(x)\leq \beta f(x)+\alpha$ for any $x\in X$. This makes $C_m^q(X)$ a partially ordered set. Similarly, we say that $[f]_c\preceq[g]_c$ if there exists a continuous function $\psi$ on $[0,\infty)$ with $\lim_{t\to\infty}\psi(t)=\infty$ such that $g(x)\leq  \psi(f(x))$ for any $x\in X$. Clearly, $[f]_q\preceq[g]_q$ implies $[f]_c\preceq[g]_c$, and $C_m^c(X)$ is a partially ordered set as well.

Set $f\land g(x)=\max(f(x),g(x))$ and $f\lor g(x)=\min(f(x),g(x))$, respectively. 


\begin{lem}
If $f,g\in\mathcal C_m(X)$ then $f\lor g,f\land g\in\mathcal C_m(X)$.

\end{lem}
\begin{proof}
The property (f1) for $f\lor g$ and for $f\land g$ is obvious. The property (f2) for them can be checked by direct calculation.

\end{proof}

\begin{lem}
Let $f'\sim f$. Then $f'\lor g\sim f\land g$ and $f'\lor g\sim f\land g$. 

\end{lem}
\begin{proof}
As the two statements are similar, we check only the first one, and only for quasi-equivalence.
If $f'(x)\leq\beta f(x)+\alpha$ for any $x\in X$ then 
\begin{eqnarray*}
\max(f'(x),g(x))&\leq& \max(\beta f(x)+\alpha,g(x))\leq \max(\beta f(x)+\alpha,\beta g(x)+\alpha)\\
&=&\beta\max(f(x),g(x))+\alpha.
\end{eqnarray*}

\end{proof}

Thus, $[f]\lor [g]=[f\lor g]$ and $[f]\land [g]=[f\land g]$ are well-defined in $C_m(X)$ (both for quasi-equivalence and for coarse equivalence).  

It is clear that $[f]\land[g]\preceq [f],[g]\preceq [f]\lor[g]$.

\begin{lem}
If $h\in\mathcal C_m(X)$ satisfies $[f],[g]\preceq [h]$ and $[h]\preceq [f]\lor [g]$ then $[h]=[f]\lor[g]$.
If $h\in\mathcal C_m(X)$ satisfies $[h]\preceq [f],[g]$ and $[f]\land [g]\preceq [h]$ then $[h]=[f]\land[g]$.

\end{lem}
\begin{proof}
Once again, it suffices to prove only the first statement and only for quasi-equivalence. By assumption, there exist $\alpha\geq 0$, $\beta\geq 1$ such that
\begin{equation}\label{11}
h(x)\leq\beta f(x)+\alpha,\quad h(x)\leq\beta g(x)+\alpha,
\end{equation}
\begin{equation}\label{12}
\min(f(x),g(x))\leq\beta h(x)+\alpha 
\end{equation}
For any $x\in X$. 
Then it follows from (\ref{11}) that $h(x)\leq\beta\min(f(x),g(x))+\alpha$, and this, together with (\ref{12}), implies $h\sim_q f\lor g$. 

\end{proof}

Thus, $C_m^q(X)$ and $C_m^c(X)$ are lattices.

Let $d\in\mathcal M(X)$. Set $F(d)=f$, where $f(x)=d(x,x')$. Clearly, this determines  maps $M^q(X)\to C_m^q(X)$ and $M^c(X)\to C_m^c(X)$.  
Consider the restriction of these maps to $E(M^q(X))$ and $E(M^c(X))$, 
\begin{equation}\label{F1}
F^q:E(M^q(X))\to C_m^q(X);\quad F^c:E(M^c(X))\to C_m^c(X).
\end{equation}

\begin{thm}\label{functions}
The maps $F^q$ and $F^c$ (\ref{F1})  are bijections.

\end{thm}
\begin{proof}
As the two cases are similar, we give the proof only for the case of quasi-equivalence. Injectivity follows from Proposition \ref{xx'}, so it remains to check surjectivity. Let $f\in\mathcal C_m(X)$. Set $A_n=\{x\in X:f(x)\leq n\}$. Then $\mathcal A=\{A_n\}_{n\in\mathbb N}$ is an expanding sequence of subsets. Indeed, if $d_X(x,A_n)\leq 1/2$ then there exists $y\in A_n$ such that $d_X(x,y)\leq 1/2$. We have $f(y)\leq n$, hence $f(x)\leq f(y)+2d_X(x,y)\leq n+1$, i.e. $x\in A_{n+1}$.

We claim that $F^q([d_{\mathcal A}]_q)=[f]_q$. This follows from the proof of Theorem \ref{equiv}, where it was shown that
for $x\in X$ satisfying $n-1<f(x)\leq n$ we have $n-1\leq d_\mathcal A(x,x')\leq n$, i.e. $|d_\mathcal A(x,x')-f(x)|\leq 1$.

\end{proof}

It is easy to see that if $s\preceq t$ in $E(M(X))$ then $F(s)\preceq F(t)$, where $F$ is either $F^q$ or $F^c$. 

Let $p,q\in E(M(X))$ be projections, and let $d_1,d_2\in\mathcal M(X)$, $[d_1]=p$, $[d_2]=q$, $d^*_i=d_i$, $i=1,2$.
For $x,y\in X$, set 
$$
b(x,y')=\max(d_1(x,y'),d_2(x,y')), 
$$
$$
c(x,y')=\inf_{u\in X}[d_X(x,u)+\min(d_1(u,u'),d_2(u,u'))+d_X(u,y)].
$$

\begin{lem}\label{lattice1}
$F([b])=F(p)\land F(q)$, $F([c ])=F(p)\lor F(q)$.

\end{lem}
\begin{proof}
As $b(x,x')=\max(d_1(x,x'),d_2(x,x'))$, the first claim is obvious. For the second claim, we have first to check that $c\in\mathcal M(X)$ (recall that unlike the maximum, the minimum of two metrics may not be a metric). The triangle inequalities for $c$ can be checked as in Lemma \ref{triangle1}. Clearly, $c^*=c$.  
The proof of Lemma \ref{proj1} can be used to show that $[c]$ is a projection. 

Note that 
\begin{equation}\label{min1}
c(x,x')=\min(d_1(x,x'),d_2(x,x')).
\end{equation} 
Indeed, taking $u=x$, we get
$$
c(x,x')=\inf_{u\in X}[2d_X(x,u)+\min(d_1(u,u'),d_2(u,u'))]\leq \min(d_1(x,x'),d_2(x,x')),
$$
and on the other hand, by the triangle inequality, we have
$$
d_i(x,x')\leq 2d_X(x,u)+d_i(u,u'), \quad i=1,2,
$$
for any $u\in X$. Hence, 
$$
\min(d_1(x,x'),d_2(x,x'))\leq [2d_X(x,u)+\min(d_1(u,u'),d_2(u,u'))]
$$
for any $u\in X$. Passing to the infimum, we get
$$
\min(d_1(x,x'),d_2(x,x'))\leq c(x,x').
$$
Thus $F([c])=F(p)\lor F(q)$.

\end{proof}

\section{Description of projections in terms of ideals}

Denote by $C_b(X)$ the commutative $C^*$-algebra of all bounded continuous complex-valued functions on $X$.
A closed ideal $J\subset C_b(X)$ is called an {\it $R$-ideal} if it has a countable approximate unit $\{u_n\}_{n\in\mathbb N}$ such that
\begin{itemize}
\item[(au1)]
$u_nu_{n+1}=u_n$ for any $n\in\mathbb N$; 
\item[(au2)]
if $u_n(x)=1$ and $u_n(y)=0$ then $d_X(x,y)> 1$.
\end{itemize}

If $\mathcal A=\{A_n\}_{n\in\mathbb N}$ is an expanding sequence then we can define a sequence of functions $u_n$, $n\in\mathbb N$, by 
$u_n(x)=\max(0,1-d_X(x,A_{2n}))$ for any $x\in X$. Then (au1)-(au2) hold, and $J=J(\mathcal A)=\operatorname{dirlim}_{n\to\infty}C_0(A_n)$ is an $R$-ideal. 

Set $G(d_{\mathcal A})=J(\mathcal A)$. It is clear that if the expanding sequences $\mathcal A$ and $\mathcal B$ define the same class in $E(M^c(X))$ then $J(\mathcal A)=J(\mathcal B)$, hence we have a well-defined map $G$ from $E(M^c(X))$ to the set of $R$-ideals of $C_b(X)$.

\begin{prop}
The map $G$ from $E(M^c(X))$ to the set of all $R$-ideals of $C_b(X)$ is one-to-one.

\end{prop}
\begin{proof}

Given an $R$-ideal $J\subset C_b(X)$, endowed by an approximate unit $\{u_n\}_{n\in\mathbb N}$ satisfying (au1)-(au2), we set $A_n=\{x\in X:u_n(x)=1\}$. If $y\in A_n$ and $d_X(x,y)\leq 1$ then $u_n(x)\neq 1$, hence, by (au1), $u_{n+1}(x)=1$, thus $x\in A_{n+1}$. Therefore, $N_1(A_n)\subset A_{n+1}$. Then $A_{2n}$ is an expanding sequence.

Let $\{u_n\}$, $\{v_n\}$, $n\in\mathbb N$, be two approximate units with (au1)-(au2) for an $R$-ideal $J$, and let 
$$
A_n=\{x\in X:u_n(x)=1\}, \quad B_n=\{x\in X:v_n(x)=1\}; 
$$
$\mathcal A=\{A_n\}_{n\in\mathbb N}$, $\mathcal B=\{B_n\}_{n\in\mathcal N}$. Let $x\in A_n$. As $v_n$ is an approximate unit, there exists $m\in\mathbb N$ such that $\|u_nv_m-u_n\|\leq 1/2$. Then 
$$
|u_n(x)v_m(x)-u_n(x)|=|v_m(x)-1|\leq 1/2, 
$$
hence $v_m(x)\neq 0$. But then (au1) implies that $v_{m+1}(x)=1$. Thus, $A_n\subset B_{m+1}$. Symmetrically, for any $n\in\mathbb N$ there exists $m\in\mathbb N$ such that $B_n\subset A_m$, hence $d_{\mathcal A}\sim_c d_{\mathcal B}$.   

\end{proof}

\begin{lem}
Let $e,f\in E(M^c(X))$, $e=[d_\mathcal A]$, $f=[d_\mathcal B]$. Then $G(e\land f)=G(e)\cap G(f)$, $G(e\lor f)=G(e)+G(f)$.

\end{lem}
\begin{proof}
Recall that the sum and the intersection of closed ideals in a $C^*$-algebra is a closed ideal as well.
Let $\{u_n\}_{n\in\mathbb N}$ and $\{v_n\}_{n\in\mathbb N}$ be approximate units for $J(\mathcal A)$ and $J(\mathcal B)$, respectively, satisfying (au1)-(au2). Set $w_n=u_nv_n$, $t_n=\min(u_n+v_n,1)$. Clearly, $w_n\in J(\mathcal A)\cap J(\mathcal B)$ and $t_n\in J(\mathcal A)+J(\mathcal B)$. We have 
$$
\{x\in X:w_n(x)=1\}=A_n\cap B_n,\quad A_n\cup B_n\subset \{x\in X:t_n(x)=1\}\subset A_{n+1}\cup B_{n+1}, 
$$
and it follows that $\{w_n\}$ and $\{t_n\}$ are approximate units in the respecting ideals. Properties (au1)-(au2) obviously hold. 

\end{proof}

\section{The dual space for $E(M(X))$}

Theorem \ref{functions}, together with Lemma \ref{lattice1}, allows to introduce the structure of a distributive lattice on $E(M^q(X))$ and on $E(M^c(X))$. Note that for general inverse semigroups, the set of idempotents is only a semi-lattice, but $E(M^{(\cdot)}(X))$ is a lattice. On the other hand, as the complement to a given projection may not exist (see Example \ref{ex1} below), $E(M^q(X))$ and $E(M^c(X))$ are not Boolean algebras. The maximal and the minimal elements in $E(M(X))$ are represented by the unit element $\mathbf 1$ and the zero element $\mathbf 0$ of $E(M(X))$ respectively. 

By \cite{MacNeille} (see also \cite{G-S}), there exists the smallest Boolean algebra $L^q(X)$ (resp., $L^c(X)$) containing the distributive lattice $E(M^q(X))$ (resp., $E(M^c(X))$). Recall its construction.
Let $H(X)$ denote all finite sums of the form $\sum_{i=1}^n e_i$, where $e_i\in E(M(X))\setminus\{0\}$, with summation mod 2, i.e. $e+e=0$ for any $e\in E(M(X))$, and with $\land$ defined by linearity. The inclusion $E(M(X))\subset H(X)$ is obvious for non-zero elements: if $e\neq \mathbf 0$ then $e$ is mapped to the sum with a single summand $e$; $\mathbf 0$ is mapped to the empty sum (the zero element of $H(X)$). Let $I$ be the ideal in $H(X)$ generated by $e+f+e\land f+e\lor f$, $e,f\in E(M(X))$, and let $L(X)=H(X)/I$. Then $L(X)$ is a Boolean algebra (the complement operation is given by $\overline{e}=\mathbf 1+e$).

Let $\widehat X^q=S(L^q(X))$ (resp., $\widehat X^c=S(L^c(X))$) denote the Stone space of the Boolean algebra $L^q(X)$ (resp., $L^c(X)$). Recall that the points of $\widehat X$ are Boolean homomorphisms from $L(X)$ to the two-element Boolean algebra $\mathbb Z/2$ (or, equivalently, ultrafilters on $L(X)$), and the topology on $\widehat X$ is the topology of pointwise convergence of nets of homomorphisms. By the Stone's representation Theorem \cite{Stone}, $L(X)$ is isomorphic to the algebra of clopen subsets of its Stone space $\widehat X$. Also, $\widehat X$ is compact, Hausdorff, and totally disconnected.

The following fact should be known to specialists, but we were unable to find a reference.

\begin{prop}
The set $\widehat X$ of Boolean homomorphisms from $L(X)$ to $\mathbb Z/2$ is canonically isomorphic to the set $\Hom(E(M(X)),\mathbb Z/2)$ of unital lattice homomorphisms from $E(M(X))$ to $\mathbb Z/2$.

\end{prop} 
\begin{proof}
The inclusion $E(M(X))\subset L(X)$ determines, by restriction, the map from $\widehat X$ to $\Hom(E(M(X)),\mathbb Z/2)$. To construct a map in the opposite direction, note that any $\varphi\in\Hom(E(M(X)),\mathbb Z/2)$ extends to a map $\bar\varphi:H(X)\to\mathbb Z/2$ by setting 
$$
\bar\varphi(e_1+\cdots+e_n)=\varphi(e_1)+\cdots+\varphi(e_n),\quad e_1,\ldots,e_n\in E(M(X)). 
$$
Then $\bar\varphi(1)=1$, and $\bar\varphi$ respects the operations $+$ and $\land$. Note that 
$$
\varphi(e\lor f)+\varphi(e\land f)+\varphi(e)+\varphi(f)=\varphi(e)\lor\varphi(f)+\varphi(e)\land\varphi(f)+\varphi(e)+\varphi(f)=0
$$ 
in $\mathbb Z/2$ for any $e,f\in E(M(X))$, i.e. $\bar\varphi|_I=0$, hence $\bar\varphi$ factorizes through $L(X)$, and the map $\check\varphi:L(X)\to\mathbb Z/2$ is a unital lattice homomorphism. As $\bar\varphi(1+f)=1+\varphi(f)$, $\check\varphi$ is also a Boolean algebra homomorphism. It is trivial to check that the two constructions are inverse to each other, hence they give a one-to-one correspondence.

\end{proof}

After the above identification, the pre-base of the topology on $\widehat X$ is given by the sets 
$$
U_e^0=\{\varphi\in\Hom(E(M(X)),\mathbb Z/2):\varphi(e)=0\}
$$ 
and 
$$
U_e^1=\{\varphi\in\Hom(E(M(X)),\mathbb Z/2):\varphi(e)=1\}, 
$$ 
$e\in E(M(X))$.

The quotient map $E(M^q(X))\to E(M^c(X))$ induces the inclusion $\widehat X^c\subset \widehat X^q$. It is easy to see that $\widehat X^q$ is much bigger. For example, if $\varphi(t)=t^{a}$, $0<a\leq 1$, then $f^a(x)=\varphi(d_X(x,x_0))$ defines a function in $\mathcal C_m(X)$, hence a metric $d^a$ on the double of $X$. Then $d^a$ is not quasi-equivalent to $d^b$ when $a\neq b$, while $d^a\sim_c d^b$ for any $a,b\in(0,1] $. Therefore, it is easier to describe the structure of $\widehat X^c$. It is still too complicated, but we are able to describe a dense subset of $\widehat X^c$ in the case when $X$ is an infinite discrete proper metric space.

Let $\omega$ be a free ultrafilter on $X$. Recall that it means that $\omega$ is a set of subsets of $X$ satisfying 
\begin{itemize}
\item[(uf1)]
$X\in\omega$;
\item[(uf2)]
any finite subset of $X$ doesn't lie in $\omega$;
\item[(uf3)]
if $A\in\omega$ and $A\subset B\subset X$ then $B\in\omega$;
\item[(uf4)]
if $A,B\in\omega$ then $A\cap B\in\omega$;
\item[(uf5)]
for any $A\subset X$ either $A$ or $X\setminus A$ lies in $\omega$.
\end{itemize}

For any infinite set $A\subset X$ there exists a free ultrafilter $\omega$ such that $A\in\omega$.
For $A\subset X$, we write $\omega(A)=1$ if $A\in\omega$, and $\omega(A)=0$ if $A\notin\omega$. We denote the set of all free ultrafilters on $X$  by $\Omega$.

Let $s\in E(M^c(X))$, and let $\mathcal A$ be an expanding sequence of subsets such that $[d_{\mathcal A}]_c=s$. Set $\tau_\omega(s)=\lim_{n\to\infty}\omega(A_n)$. This limit always exists, and it follows from Lemma \ref{le2} that if $\mathcal B$ is another expanding sequence with $[d_{\mathcal B}]_c=s$ then 
$$
\lim_{n\to\infty}\omega(A_n)=\lim_{n\to\infty}\omega(B_n), 
$$
hence 
$$
\tau_\omega:E(M^c(X))\to\{0,1\}
$$ 
is well-defined. For any free ultrafilter $\omega$, $\tau_\omega(\mathbf 1)=1$, and if $\mathcal A=\{A_n\}$, $\mathcal B=\{B_n\}$, $C_n=A_n\cup B_n$, $n\in\mathbb N$, then $\omega(C_n)=\omega(A_n)\land\omega(B_n)$. It is easy to see that $\mathcal C=\{C_n\}_{n\in\mathbb N}$ is an expanding sequence for $[d_{\mathcal A}d_{\mathcal B}]=[d_{\mathcal A}]\land[d_{\mathcal B}]$, and we have $\tau_\omega([d_\mathcal C]_c)=\tau_\omega([d_\mathcal A]_c)\land\tau_\omega([d_\mathcal B]_c)$. Thus, $\tau_\omega\in\widehat X^c$ for any $\omega\in\Omega$.

\begin{lem}\label{coarse_zero}
Let $[d]_c\in E(M^c(X))$ be zero. Then all $A_n$, $n\in\mathbb N$, are bounded.

\end{lem}
\begin{proof}
The condition $[d]_c=\mathbf 0$ means that there exists a function $\varphi:[0,\infty)\to[0,\infty)$ with $\lim_{t\to\infty}\varphi(t)=\infty$ such that $d(x,x')\geq\varphi (d_X(x,x_0))$ for some $x_0\in X$. Then $x\in A_n$ implies that $d_X(x,x_0)\leq n$.

\end{proof}  

\begin{lem}\label{dist}
Let $\varphi\in\Hom(E(M^c(X)),\mathbb Z/2)$, $\varphi(e)=0$. Then for any $R>0$ and any $n\in\mathbb N$ there exists $x\in X$ such that $d_X(x,A_n)>R$. 

\end{lem}
\begin{proof}
Assume the contrary: there exists $R>0$ and $n\in\mathbb N$ such that $d_X(x,A_n)\leq R$ for any $x\in X$, then $N_R(A_n)=X$, and $A_m=X$ for $m$ sufficiently great, which means that $e\sim_c 1$, hence $\varphi(e)=1$ --- a contradiction. 

\end{proof}

\begin{thm}
The set $\{\tau_\omega:\omega\in\Omega\}$ is dense in $\widehat X^c$.

\end{thm}
\begin{proof}
We have to check that each open subset of $\widehat X$ contains a point of the form $\tau_\omega$. Consider first the set $U_e^1$, $e\in E(M^c(X))$. We need to find $\omega$ such that $\tau_\omega(e)=\varphi(e)=1$. The condition $\varphi(e)=1$ implies that $e\neq \mathbf 0$, hence, by Lemma \ref{coarse_zero}, there is $n\in\mathbb N$ such that $A_n$ is unbounded, hence infinite. Let $\omega\in\Omega$ satisfy $\omega(A_n)=1$. Then $\tau_\omega(\mathcal A)=1$.

Second, consider the case when the open set is $U_e^0$. 
By Lemma \ref{dist}, we can choose points $x^k_n$, $k,n\in\mathbb N$, such that $d_X(x^k_n,A_k)>n$. Set $B=\{x^n_n\}_{n\in\mathbb N}$. Let $f\in E(M^c(X))$ correspond to the expanding sequence $B_n=N_n(B)$, $n\in\mathbb N$. Let $x\in N_m(B)\cap A_m$ for some $m\in\mathbb N$. 

Since $x\in N_m(B)$, there exists $n\in\mathbb N$ such that $d_X(x,x^n_n)\leq m$. There are two possibilities: $n\leq m$ and $n>m$. In the latter case, $A_m\subset A_n$, hence $d_X(x^n_n,A_m)\geq d_X(x^n_n,A_n)>n$. Since $x\in A_m$, $d_X(x^n_n,A_m)\leq d_X(x,x^n_n)\leq m$. Therefore, 
$$
m\geq d_X(x,x^n_n)\geq d_X(x^n_n,A_m)\geq d_X(x^n_n,A_n)>n.
$$
Hence, $n\leq m$. Then $d_X(x,\{x^1_1,\ldots,x^m_m\})\leq m$, hence the set $N_m(B)\cap A_m$ is bounded, hence finite.

Let $\omega\in\Omega$ be a free ultrafilter such that $\omega(B)=1$. Then $\omega(A_m)=0$ for any $m\in\mathbb N$, hence $\tau_\omega(e)=0$.

Finally, we pass to a general open set in the base of the topology of $\widehat X^c$, namely, 
$$
U_{e_1}^0\cap\cdots\cap U_{e_n}^0\cap U_{f_1}^1\cap\cdots\cap U_{f_m}^1, 
$$
where $e_1,\ldots,e_n,f_1,\ldots,f_m\in E(M^c(X))$. 
Note that $a=b=0$ and $a\lor b=0$ (resp., $a=b=1$ and $a\land b=0$) are equivalent for $a,b\in\mathbb Z/2$, hence 
$$
U_{e_1}^0\cap\cdots\cap U_{e_n}^0=U_{e_1\lor\cdots\lor e_n}^0,\quad 
U_{f_1}^1\cap\cdots\cap U_{f_m}^1=U_{f_1\land\cdots\land f_m}^1, 
$$
so it remains to consider open sets of the form $U_e^0\cap U_f^1$ for $e,f\in E(M^c(X))$.

Let $\varphi\in\Hom(E(M^c(X)),\mathbb Z/2)$ satisfy $\varphi(e)=0$ and $\varphi(f)=1$, and let $\mathcal A$ and $\mathcal B$ be the expanding sequences for $e$ and $f$ respectively. Note that $C_n=A_n\cap B_n$, $n\in\mathbb N$, is the expanding sequence for $e\land f$, and that $\varphi(e\land f)=\varphi(e)\land\varphi(f)=0$. Therefore, by Lemma \ref{dist} and the argument after that Lemma, there exists an infinite set $D$ such that $N_n(D)\cap(A_n\cap B_n)$ is finite. As $\varphi(f)=1$, there exists $m\in\mathbb N$ such that $B_m$ is infinite. Let $\omega\in\Omega$ satisfy $\omega(D)=1$ and $\omega(B_m)=1$. Then $\omega(B_n)=1$ for any $n\geq m$, and $\omega(D\cap B_n)=1$. As $A_n\cap(D\cap B_n)$ is finite, $\omega(A_n)=0$ for all $n\geq m$. Therefore, $\tau_\omega(e)=0$. As $\omega(B_m)=1$, $\tau_\omega(f)=1$. 

\end{proof}

\begin{remark}
The assignment $\omega\mapsto\tau_\omega$ defines a map $\tau:\beta X\setminus X\to\widehat X^c$, where $\beta X$ is the Stone--\v Cech compactification of $X$. Regretfully, this map is not continuous. Although 
$$
\tau^{-1}(U^1_e)=\{\omega\in\Omega:\omega(A_n)=1 \mbox{\ for\ some\ }n\}
$$ 
is open, but 
$$
\tau^{-1}(U^0_e)=\{\omega\in\Omega:\omega(A_n)=0 \mbox{\ for\ any\ }n\}=\cap_{n\in\mathbb N}\{\omega\in\Omega:\omega(A_n)=0\}
$$ 
need not be open.
We also don't know any condition allowing to determine when $\omega_1,\omega_2\in\Omega$ satisfy $\tau_{\omega_1}=\tau_{\omega_2}$ (but see Example \ref{ex2}).    

\end{remark}

\section{Measures on $\widehat X$}

Let $Cl(X)$ denote the set of closed subsets of $X$.
We call a map $\mu:Cl(X)\to[0,1]$ an {\it admissible measure on $X$} if
\begin{itemize}
\item[(am1)]
$\mu$ is a finitely additive probability measure;
\item[(am2)]
$\mu(K)=0$ for any bounded closed set $K\subset X$.
\end{itemize}

Note that if $X$ is discrete and proper then (am2) means that $\mu$ vanishes on finite subsets.

Let $a\in E(M(X))$, and let $\mathcal A=\{A_n\}_{n\in\mathbb N}$ be its expanding sequence. For an admissible measure $\mu$ on $X$, set $\hat\nu_\mu(a)=\lim_{n\in\mathbb N}\mu(A_n)$. As $A_n\subset A_{n+1}$, the sequence $\mu(A_n)$ is monotonely non-decreasing and bounded from above by $\mu(X)=1$, hence the limit exists. Moreover, if $\mathcal B=\{B_n\}$ is another expanding sequence for $a$ then there is a monotone function $\varphi$ on $[0,\infty)$ with $\lim_{t\to\infty}\varphi(t)=\infty$ (if we are working with quasi-equivalence then this function is linear) such that $A_n\subset B_{\varphi(n)}$ and $B_n\subset A_{\varphi(n)}$, hence $\lim_{t\to\infty}\mu(A_n)=\lim_{t\to\infty}\mu(B_n)$. Thus the map $\hat\nu_\mu:E(M(X))\to[0,1]$ is well-defined. 

\begin{lem} The map $\hat\nu_\mu$ has the following properties:

\begin{itemize}
\item[(n1)]
$\hat\nu_\mu(\mathbf 0)=0$, $\hat\nu_\mu(\mathbf 1)=1$;
\item[(n2)]
$\hat\nu_\mu(a)\leq\hat\nu_\mu(b)$ if $a\preceq b$, $a,b\in E(M(X))$;
\item[(n3)]
$\hat\nu_\mu(a\land b)+\hat\nu_\mu(a\lor b)=\hat\nu_\mu(a)+\hat\nu_\mu(b)$ for any $a,b\in E(M(X))$.

\end{itemize}
\end{lem}
\begin{proof}
Let $d\in\mathcal M(X)$. If $[d]=\mathbf 1$ then $A_n=X$ for sufficiently great $n$, hence $\hat\nu_\mu(\mathbf 1)=1$. If $[d]=\mathbf 0$ then $A_n$ is bounded for any $n\in\mathbb N$, hence $\mu(A_n)=0$, therefore $\hat\nu_\mu(\mathbf 0)=0$. (n2) is obvious.

If $\mathcal A=\{A_n\}$ and $\mathcal B=\{B_n\}$, $n\in\mathbb N$, are expanding sequences for $a$ and $b$, respectively, then the expanding sequences for $a\land b$ and $a\lor b$ are $\{A_n\cap B_n\}$ and $\{A_n\cup B_n\}$, respectively. As $\mu(A_n\cap B_n)+\mu(A_n\cup B_n)=\mu(A_n)+\mu(B_n)$, passing to the limit, we get (n3).

\end{proof}

Let $L^c(X)$ denote the Boolean algebra generated by $E(M^c(X))$ as above. By Smiley--Horn--Tarski Extension Theorem (\cite{Konig}, Theorem 3.4; cf. \cite{Smiley}, Section 1), $\hat\nu_\mu$ extends to a function $\bar\nu_\mu:L^c(X)\to[0,1]$ by the formula 
$$
\bar\nu_\mu(e_1+\cdots e_n)=\sum_{i=1}^n(-2)^{i-1}\sigma_{i}(e_1,\ldots,e_n), 
$$
where 
$$
\sigma_{i}(e_1,\ldots,e_n)=\sum\hat\nu_\mu(e_{j_1}\land\cdots\land e_{j_i}), 
$$
and the summation extends over all properly monotone $i$-tuples $j_1<\cdots<j_i$ of the integers $1,2,\ldots,n$, and $\bar\nu_\mu$ is an additive measure, i.e. it satisfies the following properties:
\begin{itemize}
\item[(m1)]
$\bar\nu_\mu(\mathbf 1)=1$;
\item[(m2)]
$\bar\nu_\mu(\overline{a})=1-\bar\nu_\mu(a)$ for any $a\in L^c(X)$, where $\overline{a}=\mathbf 1+a$ is the complement of $a$;
\item[(m3)]
$\bar\nu_\mu(a\land b)+\bar\nu_\mu(a\lor b)=\bar\nu_\mu(a)+\bar\nu_\mu(b)$ for any $a,b\in L^c(X)$.

\end{itemize} 

The following remark explains why, from the point of view of measure theory, the case of coarse equivalence is more interesting that the case of quasi-equivalence.
 
\begin{remark}
Let $d\in\mathcal M(X)$, $[d]_q\neq \mathbf 0$, $[d]_c=\mathbf 0$, and let $\mathcal A=\{A_n\}_{n\in\mathbb N}$ be the extending sequence for $d$. Then all $A_n$, $n\in\mathbb N$, are bounded, hence $\bar\nu_\mu([d]_q)=0$. On the other hand, if $[d]_c\neq \mathbf 0$ then all $A_n$ for sufficiently great $n$ are infinite, and there exists an admissible measure $\mu$ on $X$ such that $\mu(A_n)\neq 0$, hence $\bar\nu_\mu([d]_c)\neq 0$.  

\end{remark}

For $a\in L^c(X)$, let $U_a\subset\widehat X$ be the clopen set corresponding to $a$ under the Stone duality, i.e. the characteristic function of $U_a$ is $a$. Then $U_a\cap U_b=U_{a\land b}$ and $U_a\cup U_b=U_{a\lor b}$ for any $a,b\in L^c(X)$. Recall that $L^c(X)$ is the collection of the characteristic functions of all clopen sets of $\widehat X^c$. We may take infinite unions and intersections of clopen sets. If $a,a_1,a_2,\ldots\in L^c(X)$ satisfy $a_i\land a_j=0$ when $i\neq j$, and $a=\lor_{i=1}^\infty a_i$ then only finite number of $a_i$, $i\in\mathbb N$, can be non-zero (equivalently, only finite number of $U_{a_i}$ can non-empty) (cf. \cite{Kakutani}): if $a\in L^c(X)$ then $U_a$ is a closed subset of the compact space $\widehat X^c$, while $U_{a_i}$, $i\in\mathbb N$, is a disjoint open cover of $U_a$, hence there exist $a_{i_1},\ldots, a_{i_n}$ such that $U_a=\cup_{j=1}^n U_{a_{i_j}}$, therefore, $U_{a_i}=\emptyset$ if $i\neq i_1,\ldots, i_n$. Then, by Hahn–-Kolmogorov Theorem (\cite{Tao}, Theorem 1.7.8), $\bar\nu_\mu$ can be extended to a unique sigma-additive meaure $\nu_\mu$ on the sigma-algebra generated by $L^c(X)$, i.e. by clopen sets of $\widehat X^c$.

The next Lemma shows that the set of projections of type I is sufficiently large.

\begin{lem}\label{measure0}
Let $a\in E(M(X))$, $a=[d_{\mathcal A}]$ with the expanding sequence $\mathcal A=\{A_n\}$, and let $a_n=[d_{\mathcal E_{A_n}}]$, $n\in\mathbb N$. Then  $\cup_{n\in\mathbb N}U_{a_n}\subset U_a$, and $\nu_\mu(U_a)=\nu_\mu(\cup_{n\in\mathbb N}U_{a_n})$ for any admissible measure $\mu$ on $X$.

\end{lem}
\begin{proof}
It follows from the definition of the metrics $d_{\mathcal E_{A_n}}$, $d_{\mathcal A}$ that 
$$
d_{\mathcal E_{A_1}}(x,x')\geq d_{\mathcal E_{A_2}}(x,x')\geq \cdots\geq d_{\mathcal A}(x,x') 
$$
for any $x\in X$, hence $a_1\preceq a_2\preceq\cdots\preceq a$ and, equivalently, $U_{a_1}\subset U_{a_2}\subset\cdots\subset U_a$.

The inclusion $A\subset N_k(A)$, $k\in\mathbb N$, implies that 
$$
\mu(A_n)\leq \nu_\mu(U_{a_n})=\lim_{k\to\infty}\mu(N_k(A_n))=\sup_{k\in\mathbb N}\mu(N_k(A)). 
$$
Then sigma-additivity of $\nu_\mu$ implies that 
$$
\nu_\mu(U_a)=\lim_{n\to\infty}\mu(A_n)=\sup_{n\in\mathbb N}\mu(A_n)\leq \sup_{n\in\mathbb N}\nu_\mu(U_{a_n})=\nu_\mu(\cup_{n\in\mathbb N}U_{a_n}).
$$

\end{proof}

Let $E_0(M(X))\subset E(M(X))$ denote the set of all projections of type I. It is not a lattice, as the meet $e\land f$ of two projections, $e$ and $f$, of type I may be of type II, but it is a join semilattice: if $e=[d_{\mathcal E_A}]$, $f=[d_{\mathcal E_B}]$ then $e\lor f=[d_{\mathcal E_{A\cup B}}]$. Let $L_0(X)\subset L(X)$ denote the Boolean subalgebra generated by $E_0(M(X))$, let $\widehat X_0$ be the Stone dual space for $L_0(X)$, and let $\pi:\widehat X\to\widehat X_0$ be the canonical surjective map induced by the inclusion $L_0(X)\subset L(X)$. 
 
Let $\Sigma$ (resp., $\Sigma_0$) denote the sigma-algebra on $\widehat X$ (resp., on $\widehat X_0$) generated by all clopen sets, and let $\Sigma'$ be the sigma-algebra on $\widehat X$ of sets $\pi^{-1}(U)$, $U\in\Sigma_0$. For a sigma-algebra $\Psi$ on $Y$, let $\mathcal F(Y,\Psi)$ denote the space of all measurable functions on $Y$. 

\begin{prop}\label{measure1}
Let $U\in\Sigma$. Then there exists $V\in\Sigma'$ such that $V\subset U$ and $\nu_\mu(V)=\nu_\mu(U)$ for any admissible measure $\mu$ on $X $.

\end{prop}
\begin{proof}
This directly follows from Lemma \ref{measure0}.

\end{proof}

\begin{cor}
The spaces $\mathcal F(\widehat X,\Sigma)$, $\mathcal F(\widehat X,\Sigma')$ and $\mathcal F(\widehat X_0,\Sigma_0)$
are canonically isomorphic.

\end{cor}
\begin{proof}
It follows from Proposition \ref{measure1} that if a function $f$ on $\widehat X$ is measurable with respect to $\Sigma$ then there exists a function $g$ on $\widehat X$, measurable with respect to $\Sigma'$, and $\nu_\mu(\{x\in\widehat X:f(x)\neq g(x)\})=0$, hence the canonical inclusion $\mathcal F(\widehat X,\Sigma')\subset \mathcal F(\widehat X,\Sigma)$ is an isomorphism.

Let $f$ be a measurable function on $\widehat X$ with respect to $\Sigma'$, let $x_1,x_2\in\widehat X$ such that $\pi(x_1)=\pi(x_2)$, and let $f(x_1)<c<f(x_2)$. There exists $U\in\Sigma_0$ such that $\{x\in\widehat X:f(x)<c\}=\pi^{-1}(U)$. Then $x_1\in\pi^{-1}(U)$, $x_2\notin\pi^{-1}(U)$, which contradicts $\pi(x_1)=\pi(x_2)$, hence any measurable function on $\widehat X$ with respect to $\Sigma'$ is constant on $\pi^{-1}(y)$ for any $y\in\widehat X_0$, hence the map $\mathcal F(\widehat X_0,\Sigma_0)\to\mathcal F(\widehat X,\Sigma')$ induced by $\pi$ is an isomorphism.  

\end{proof}

\section{Examples}

\begin{example}\label{typeI}
Let $\varphi:\mathbb N\to\mathbb N$ be a map that takes each value infinitely many times and satisfies $\varphi(n)\leq n$ for any $n\in\mathbb N$. Let $x_n^\pm=(n^2,\pm\varphi(n))$, and let $X=\{x_n^\pm:n\in\mathbb N\}\subset \mathbb Z^2$ with the Manhattan metric. We write $n(x)=n$ when $x=x_n^\pm$. Let $A^\pm=\{x_n^\pm:n\in\mathbb N\}$. Let $b_\pm(x,y')=d_X(x,A^\pm)+d_X(y,A^\pm)+1$. Then the equivalence class $[b_+]_c[b_-]_c$ is represented by the metric 
$b\in\mathcal M(X)$ such that 
\begin{eqnarray*}
b(x,x')&=&\inf_{y\in X}[d_X(x,A^+)+d_X(y,A^+)+d_X(x,A^-)+d_X(y,A^-)+2]\\
&=&d_X(x,A^+)+d_X(x,A^-)+4.
\end{eqnarray*}

Let $\mathcal C=\{C_n\}_{n\in\mathbb N}$ be the expanding sequence for the metric $b$. For $x=x_n^+$ we have $d_X(x,A^+)=0$ and $d_X(x,A^-)=2\varphi(n)$, while for $x=x_n^-$ we have $d_X(x,A^+)=2\varphi(n)$ and $d_X(x,A^-)=0$. 
Then 
$$
C_m=\{x\in X:2\varphi(n(x))\leq m-4\},\quad C_{m+2}=\{x\in X:2\varphi(n(x))\leq m-2\}. 
$$
Let $m$ be even, and let $n$ satisfy $\varphi(n)=\frac{m-2}{2}$. Then $x_n^\pm\in C_{m+2}\setminus C_m$. There are infinitely many such $n$'s. If $y\in C_m$ then $n(y)\neq n(x)$, hence $d_X(x_n^\pm,C_m)\geq |n^2-n(y)^2|\geq n$, thus there is no $k$ such that $C_{m+2}\subset N_k(C_m)$, hence the projection $[b]_c$ is of type II, being the product of type I projections. 

\end{example}

\begin{example}\label{ex1}
Let $X=\mathbb N$ with the standard metric, let $A=\{2^k:k\in\mathbb N\}\subset X$, and let $b=d_{\mathcal E_A}\in\mathcal M(X)$. We claim that $[b]$ is not complementable in $E(M(X))$, i.e. there is no $e\in E(M(X))$ such that $e\land [b]=0$ and $e\lor[b]=1$. It suffices to prove this claim for the case of coarse equivalence, as $M^c(X)$ is the quotient of $M^q(X)$. 

Suppose that such $e$ exists, and let $\mathcal B=\{B_n\}_{n\in\mathbb N}$ be an expanding sequence such that $e=[d_{\mathcal B}]$. The condition $[d_{\mathcal B}]_c\lor[b]_c=1$ implies that there exists $m\in\mathbb N$ such that $N_{m/2}(A)\cup B_m=X$. Then $B_m$ contains all points between $2^k+m/2$ and $2^{k+1}-m/2$ for all $k$ such that $2^k>m$. As $N_{n/2}(B_m)\subset B_{n+m}$, the set $B_{2n+1}$ contains all the points of the form $2^k$, where $k>\log_2 m$, i.e. $A\cap B_{2n+1}$ is infinite, hence $N_k(A)\cap B_{2n+1}$ is infinite too, and by Lemma \ref{0}, $[b]_c[d_{\mathcal B}]_c\neq 0$.

\end{example}

\begin{example}\label{ex2}
Let $X=\{2^n:n\in\mathbb N\}$ with the standard metric. An important feature of this case is that the sets $A$ and $N_k(A)$ differ by only a finite number of points for any $k\in\mathbb N$ and for any $A\subset X$. Let $b_A=d_{\mathcal E_A}\in\mathcal M(X)$. The set of all $[b_A]_c$, $A\subset X$, forms a sublattice $L_0(X)$ in $E(M^c(X))$. Let $\varphi\in\Hom(E(M^c(X)),\mathbb Z/2)$. Set $\omega(A)=\varphi([b_A]_c)$ (here $A\neq\emptyset$, so we have to set, additionally, $\omega(\emptyset)=0$). Note that if $[b_A]_c=[b_B]_c$ then $A$ and $B$ differ by a finite number of points. Indeed, $b_A\sim_c b_B$ implies that there exists a monotone function $\psi$ on $[0,\infty)$ with $\lim_{t\to\infty}\psi(t)=\infty$ such that $d_X(x,A)\leq \psi(d_X(x,B))$ and $d_X(x,B)\leq \psi(d_X(x,A))$. If $x\in B$ then $d_X(x,A)\leq\psi(0)$ and if $x\in A$ then $d_X(x,B)\leq\psi(0)$, hence there exists $k\in\mathbb N$ such that $A\subset N_k(B)$ and $B\subset N_k(A)$. 

Note that in this case $L_0(X)$ is not only a lattice, but also a Boolean algebra. Indeed, let $B=X\setminus A$. Then $N_k(A)\cap N_k(B)$ is finite for any $k$. Therefore, $[b_A]_c\land [b_B]_c=0$ and $[b_A]_c\lor [b_B]_c=1$. Therefore, $\omega(B)=1-\omega(A)$. Also we have $\omega(X)=1$, $\omega(C)=0$ for any finite set $C\subset X$. Other properties of a free ultrafilter are obviously true, so $\omega$ is a free ultrafilter on $X$. 

Let $\widehat X_0=S(L_0(X))$ denote the Stone space of the Boolean algebra $L_0(X)$. The restriction determines a map $\pi:\widehat X^c\to\widehat X_0$ by $\pi(\varphi)= \varphi|_{L_0(X)}$. 
Let $A\subset X$. Then 
$$
\tau_\omega|_{L_0(X)}([b_A]_c)=\tau_\omega([b_A]_c)=\lim_{n\to\infty}\omega(N_n(A))=\omega(A),
$$ 
i.e. $\tau_\omega|_{L_0(X)}=\omega$, hence the map $\pi$ is surjective.

\end{example}

\end{document}